\documentclass[11pt]{article}
\usepackage{a4wide}

\usepackage[a4paper, margin=2.3cm]{geometry}
\usepackage{amsmath,amssymb,amsthm, comment, enumitem, array}
\usepackage{color}
\usepackage{svg}
\usepackage{parskip}
\usepackage{hyperref}
\usepackage{parskip}
\usepackage{setspace}
\usepackage{graphicx}
\usepackage{mathtools}
\usepackage[thinc]{esdiff}
\usepackage[euler]{textgreek}

\DeclareMathOperator{\Ann}{Ann}
\usepackage{babel}

\usepackage[capitalise, nameinlink]{cleveref}
\crefname{equation}{}{}
\crefname{figure}{{\sc Figure}}{{\sc Figure}}
\crefname{subsection}{Subsection}{Subsections}

\newtheorem{theorem}{Theorem}[section]
\newtheorem{proposition}[theorem]{Proposition}

\newtheorem{lemma}[theorem]{Lemma}

\newtheorem*{definition*}{Definition}

\def\Z{{\mathbb  Z}}
\def\R{{\mathbb  R}}
\def\N{{\mathbb  N}}

\hypersetup{
    colorlinks=true,
    linkcolor = blue,
    citecolor=magenta,
    urlcolor=cyan,
}

\def\R{\mathbb{R}}

\def \R{{\mathbb R}}

\def \Z{{\mathbb Z}}
\def \1{\textbf{1}}

\def\Q{\mathbb{Q}}

\def\l({\left(}
\def\r){\right)}

\newcommand{\dist}{\operatorname{dist}}
\newcommand{\fr}[1]{\left\{#1\right\}} 
\newcommand{\nnorm}[1]{\left\|#1\right\|}

\theoremstyle{definition}

\newcommand{\T}{\mathbb T}

\begin{document}
\title{Polynomial extensions of Raimi's theorem}
\author{Norbert Hegyv\'{a}ri\thanks{E\"otv\"os University and associated member of Alfr\'{e}d R\'{e}nyi Institute, Hungary. Email: hegyvari@renyi.hu} \and J\'{a}nos Pach\thanks{Alfr\'{e}d R\'{e}nyi Institute, Hungary. Email:
{\tt pach@renyi.hu}. \newline
\hspace*{0.5cm} Supported by NKFIH grants K-131529,  and ERC Advanced Grant 882971``GeoScape.''}\and Thang Pham\thanks{Institute of Mathematics and Interdisciplinary Sciences at Xidian University, China. \newline
\hspace*{0.5cm} Email: phamanhthang.vnu@gmail.com}}
\maketitle
\begin{abstract}
Raimi's theorem guarantees the existence of a partition of $\mathbb{N}$ into two parts with an un\-avoid\-able intersection property: for any finite coloring of $\mathbb{N}$, some color class intersects both parts infinitely many times, after an appropriate shift (translation). We establish a polynomial extension of this result, proving that such intersections persist under polynomial shifts in any dimension. Let $P^{(1)},\dots,P^{(f)}\in\mathbb{Z}[x]$ be non-constant polynomials with positive leading coefficients and $P^{(j)}(0)=0$ for every $j$. We construct a partition of $\mathbb{N}^k$ into an arbitrarily fixed finite number of pieces such that for any coloring of $\mathbb{N}^k$ with finitely many colors, there exist $x_0\in \mathbb{N}$ and a single color class that meets all partition pieces after shifts by $x_0+P^{(j)}(h)$ in each of the $k$ coordinate directions, for every $j$ and infinitely many values $h\in \mathbb{N}$. Our proof exploits Weyl's equi\-dis\-tri\-bu\-tion theory, Pontryagin duality, and the structure of polynomial relation lattices. We also prove some finite analogues of the above results for abelian groups and $SL_2(\mathbb{F}_q)$.
\end{abstract}

\tableofcontents
\section{Introduction}

Ramsey theory admits two complementary perspectives. The \emph{coloring approach} begins with an arbitrary finite partition of the integers and seeks structures that some color class must contain, for instance, Van der Waerden's theorem \cite{vanderWaerden1927} guarantees arbitrarily long monochromatic arithmetic progressions. The \emph{density approach} assumes only that a set has positive upper density and proves the occurrence of prescribed patterns. Szemer\'edi \cite{Sze} established the existence of arbitrarily long arithmetic progressions in sets of positive density, while S\'ark\"ozy \cite{Sarkozy1978} and independently Furstenberg \cite{Furstenberg1977} proved that such sets contain pairs whose difference is a nonzero square.

Raimi \cite{Raimi} introduced a complementary viewpoint that begins with neither a given dense set nor a given coloring. Instead, one constructs a specific partition of $\mathbb{N}$ with an unavoidable intersection property: for every finite coloring, there exists a shift such that a single color meets every piece of the partition in an infinite set. Informally, while classical Ramsey results ask ``what patterns are forced by density or by finite colorings?'', Raimi's perspective asks ``which partitions are unavoidable in any finite coloring after an appropriate shift?''

Raimi's theorem reads as follows.

\begin{theorem}[Raimi \cite{Raimi}]\label{raimithm}
There exists a subset $E\subseteq \mathbb{N}$ with the following property. For every finite coloring of $\mathbb{N}$ with $t$ colors, $\mathbb{N}=\bigcup_{j=1}^t F_j$, $t\in\mathbb{N},$ there exist a single color class $F_m$, $1\le m\le t$, and $k\in \mathbb{N}$ such that both $(F_m+k)\cap E$ and $(F_m+k)\setminus E$ are infinite sets.
\end{theorem}

In other words, after an appropriate shift $k$, the set $F_m + k$ unavoidably intersects both $E$ and its complement in infinitely many elements. This theorem was originally proved by Raimi using topological methods, and subsequently Hindman \cite{HM} provided an elementary proof. Strengthened versions specifying the densities of the partition sets or guaranteeing positive densities in the conclusion can be found in \cite{Bergelson2, NH}.

Throughout this paper, $\mathbf{1}_k$ stands for the element of $\mathbb{N}^k$ whose each coordinate is $1$. For any finite or infinite set $X$ of integers (or elements of a group equipped with a binary operation $+$), let $\mathrm{FS}(X)$ denote the set of all finite sums formed by the elements of $X$.

Raimi's theorem extends naturally to higher-dimensional lattices using diagonal shifts. 

\begin{theorem}\label{HD}
Let $r, t, k\in\mathbb{N}$. There exists a partition
$\bigcup_{i=1}^rE_i$ of $\mathbb{N}^k$ such that 
for every finite coloring of $\mathbb{N}^k$ with $t$ colors,
$
\mathbb{N}^k = \bigcup_{j=1}^t F_j,
$ there exist a single color class $F_m$, $1\le m\le t$, an element $x_0\in \mathbb{N}$, and a sequence $\{x_n\}_{n\ge 1}$
in $\mathbb{N}$ such that for every $h$ from the set $x_0+
\mathrm{FS}(\{x_n\}_{n\ge 1})$ and every $i\in\{1,2,
\ldots,r\}$, the set $(F_m+h\mathbf{1}_k)\cap E_i$ has infinitely many elements.
\end{theorem}

While Theorem~\ref{HD} extends the intersection property to higher dimensions with diagonal shifts, a natural question arises: does this property persist under \emph{polynomial} shifts? This question is motivated by Bergelson and Leibman's celebrated extension \cite{BL} of Szemer\'edi's theorem from arithmetic progressions to polynomial progressions. Their work demonstrated that polynomial patterns are just as robust as linear patterns in dense sets. 

Our main result establishes an analogous phenomenon for Raimi-type intersections, showing that partition unavoidability extends to polynomial shifts and revealing deep connections to Weyl's polynomial equidistribution theory.

\begin{theorem}\label{thmpolynomial-main}
Let $r, t, k, f\in \mathbb{N}$.
Let $P^{(1)},\dots,P^{(f)}\in\Z[x]$ be non-constant polynomials with the properties that $P^{(1)}(0)=\cdots=P^{(f)}(0)=0$ and the leading coefficients are positive. 

There exists a partition
$\bigcup_{i=1}^rE_i$ of $\mathbb{N}^k$ such that for every finite coloring of $\mathbb{N}^k$ with $t$ colors,
$
\mathbb{N}^k = \bigcup_{j=1}^t F_j,
$ there exist a single color class $F_m$, $1\le m\le t$, an element $x_0\in \mathbb{N}$, and a set $H\subset\mathbb{N}$ of positive lower density with the property that for every $h\in H$ and every $j\in\{1,\dots,f\}$, the set
\[
\big(F_m+(x_0+P^{(j)}(h))\mathbf{1}_k\big)\cap E_i
\]
has infinitely many elements.
\end{theorem}


While Theorem \ref{thmpolynomial-main} extends the intersection property to polynomial shifts, it does not formally generalize Theorem \ref{HD}. The latter provides shifts forming an FS-set with additive structure, while Theorem \ref{thmpolynomial-main} guarantees infinitely many individual shifts that simultaneously work for all polynomial shifts $P^{(j)}(h)$, $1\le j\le f$.
\smallskip

Beyond the infinite setting, we can ask whether similar phenomena occur in finite algebraic structures. We answer this affirmatively, establishing finite analogues in both abelian groups and the non-abelian group $SL_2(\mathbb{F}_q)$. These results demonstrate that Raimi-type intersection properties extend to finite algebraic structures, with quantitative bounds replacing infinite intersections.

\begin{theorem}\label{thm3.1}Let $r, t\in \mathbb{N}$. Let $G$ be a finite cyclic group of order $N$, written additively. There are constants \(\alpha=\alpha(r, t)>0\) and \(N_0=N_0(r,t)\) such that the following holds for every \(N\ge N_0\). There exists a partition $\bigcup_{i=1}^rE_i$ of $G$ such that for every finite coloring of $G$ with $t$ colors,
$
G = \bigcup_{j=1}^t F_j,
$ there exist a single color class $F_m$, $1\le m\le t$, and an element $h\in G$ such that
\[\forall i\in\{1,2,
\ldots,r\},\quad |(F_m+h)\cap E_i|\ge \alpha |G|.\]
The same conclusion holds for any finite abelian group $G$ whose exponent is at least $N_0$, albeit with a weaker constant $\alpha'(r,t)\ge \frac{\alpha(r,t)}{t}.$
\end{theorem}

If the finite abelian group $G$ is written multiplicatively, then, of course, the conclusion should read as $|(F_m\cdot h)\cap E_i|\ge \alpha |G|$. 


\begin{theorem}\label{thm3.2}
Let $r, t\in \mathbb{N}$. There are constants \(\alpha=\alpha(r, t)>0\) and \(q_0=q_0(r,t)\) such that the following holds for every odd prime power \(q>q_0\). There exists a partition
$\bigcup_{i=1}^rE_i$ of $SL_2(\mathbb{F}_q)$
such that for every finite coloring of
$SL_2(\mathbb{F}_q)$ with $t$ colors,
$
SL_2(\mathbb{F}_q) = \bigcup_{j=1}^t F_j,
$
there exist a single color class $F_m$, $1\le m\le t$, and an element \(h\in SL_2(\mathbb{F}_q)\) such that
\[
\forall\, i\in\{1,\ldots, r\},\quad |(F_m\cdot h)\cap E_i|\ge \alpha\,|SL_2(\mathbb{F}_q)|.
\]
\end{theorem}

The above results illustrate that Raimi-type intersection properties are robust across diverse mathematical structures, from infinite lattices under polynomial shifts to finite groups, both abelian and non-abelian.

Another closely related topic in this area is Hindman's theorem \cite{Hind}, which asserts that for any finite coloring of $\mathbb{N}$ there exists an infinite set $X \subset \mathbb{N}$ such that $\mathrm{FS}(X)$ is monochromatic. Recent work on density refinements of Hindman's theorem can be found in \cite{HKR,Kra}.

It would be interesting to explore whether Theorem \ref{thm3.1} and Theorem \ref{thm3.2} can be applied to other geometric problems in the finite field setting. For instance, they might have implications for the Erdős–Falconer distance problem \cite{hart, IR} or for understanding expansion phenomena in the group $SL_2(\mathbb{F}_q)$ \cite{chapmanI, HH}.

\subsection{Proof ideas and techniques}

Our proofs employ a variety of methods from combinatorics, harmonic analysis, and equidistribution theory. 

For Theorem \ref{HD}, we construct a partition of $\mathbb{N}^k$ using a coloring based on $\mathbb{Q}$-linearly independent real numbers $\alpha_1, \ldots, \alpha_k$. The key is to define $\varphi(\mathbf{a}) = \{\sum_i \alpha_i a_i\} \in (0,1)$ and color points based on which interval in a dyadic partition of $(0,1)$ contains $\varphi(\mathbf{a})$. By choosing $\beta = \alpha_1 + \cdots + \alpha_k$ irrational, we can apply Lemma \ref{lm2.2} to find a sequence whose finite sums land in prescribed narrow intervals modulo $1$, ensuring the desired intersection property.

Theorem \ref{thmpolynomial-main} requires significantly deeper machinery. The main step is to simultaneously control the fractional parts $\{\beta P^{(j)}(h)\}$ for all $j = 1, \ldots, f$. While Weyl's classical theorem handles a single polynomial, we need Theorem \ref{lm2.2'}, which guarantees a set of positive lower density satisfying these constraints simultaneously. The proof relies on Weyl's equidistribution criterion (Theorem \ref{pro31}), Pontryagin duality for locally compact abelian groups (Proposition \ref{lem:annihilator-ker}), and careful analysis of the relation lattice $R$ and the associated subtorus $H \subseteq \mathbb{T}^f$. We show that the sequence $\{\mathbf{v}(n)\}_{n\ge 1}$, $\mathbf{v}(n)= (\{\beta P^{(1)}(n)\}, \ldots, \{\beta P^{(f)}(n)\})$, is equidistributed on $H$ (Lemma \ref{prop:weyl-H}), which provides many choices of $h$ with the required properties.

For the finite group results, the techniques shift to combinatorics and probability. 

Since $G$ is a finite cyclic group, it is sufficient to consider the case $G=\mathbb{Z}_N$. Theorem \ref{thm3.1} uses an iterative refinement argument: we partition $\mathbb{Z}_N$ into nested intervals of decreasing size, apply a density-averaging at each stage, and carefully track how intersections propagate through the hierarchy. 

Theorem \ref{thm3.2} employs a probabilistic method based on Chebyshev's inequality. We project $SL_2(\mathbb{F}_q)$ onto $\mathbb{F}_q^2$, then apply Lemma \ref{lm5.22}, which shows that for any sufficiently large subset $U \subseteq \mathbb{F}_q^2$, there exists a direction such that most lines in that direction contain many points from $U$. We then lift the configuration back to the group using a carefully chosen element $h \in SL_2(\mathbb{F}_q)$. The key is that most lines contain a substantial portion of the projected set, which translates to large intersections in $SL_2(\mathbb{F}_q)$. 

{\bf Notation.} By $\{x_n\}_{n\ge 1}$ we mean a sequence with infinitely many elements $x_1, x_2, \ldots$, and if the sequence has $k$ elements, then we write $\{x_n\}_{n=1}^k$. For any $x\in \mathbb{R}$, let $\{x\}$ stand for the fractional part of $x$. For $A\subseteq\N$, we define the \emph{lower density} of $A$  as $\underline
d(A)=\displaystyle\liminf_{N\to \infty}{|A\cap[1, N]|\over N}$, and the
\emph{upper density} of $A$ as $\overline d(A)=\displaystyle \limsup_{N\to
\infty}{|A\cap[1, N]|\over N}$. If $\underline d(A)=\overline d(A)$, then we say the \emph{density} of $A$ exists and denote it by $d(A):=\underline d(A)=\overline d(A)$.

\textbf{Organization.} The rest of the paper is organized as follows. In Section~2, we establish Theorem~\ref{HD} (higher dimensions with diagonal shifts). In Section~3, we prove Theorem~\ref{thmpolynomial-main} (polynomial shifts). The proof of the key equidistribution result (Theorem~\ref{lm2.2'}) is presented in Subsections~3.1--3.2. Sections~4 and~5 handle the finite cyclic case (Theorem~\ref{thm3.1}) and the non-abelian case \(SL_2(\mathbb{F}_q)\) (Theorem~\ref{thm3.2}), respectively.

\section{Higher dimensions -- Proof of Theorem \ref{HD}}
\subsection{Partition and an induced coloring}
Fix integers $r\ge2$ and $k\ge1$, and real numbers $\alpha_1,\dots,\alpha_k$ such that
$1,\alpha_1,\dots,\alpha_k$ are $\mathbb Q$-independent.
Such numbers exist, say $\alpha_i=\log_2 p_{i+1}$ where $p_2=3<p_3=5<\dots$ are the prime numbers, and $\log_2(\cdot)$ means a logarithm in base two.

Write, for $\mathbf{a}=(a_1,\dots,a_k)\in\N^k$,
\[
\varphi(\mathbf{a}):=\Big\{\sum_{m=1}^k \alpha_m a_m\Big\}\in(0,1),\qquad
\beta:=\alpha_1+\cdots+\alpha_k\notin\mathbb Q.
\]

For $j\ge0$, set
$
S_j:=(1-2^{-j},\,1-2^{-(j+1)}]\subset(0,1),\quad
a_j:=1-2^{-j},\quad w_j:=|S_j|=2^{-(j+1)}.
$

Define
\[
\tau_j:=\frac{|S_j|}{r}=\frac{w_j}{r}.
\]
For $i\in\{1,\dots,r\}$, define the tile
\[
T_{j,i}
:=\Bigl(a_j+(i-1)\,\tau_j,\ \ a_j+i\,\tau_j\Bigr].
\]
These $r$ tiles $T_{j, i}$ form a partition of $S_j$.

\begin{figure}[htbp]
  \centering
  \includegraphics[width=1.0\textwidth]{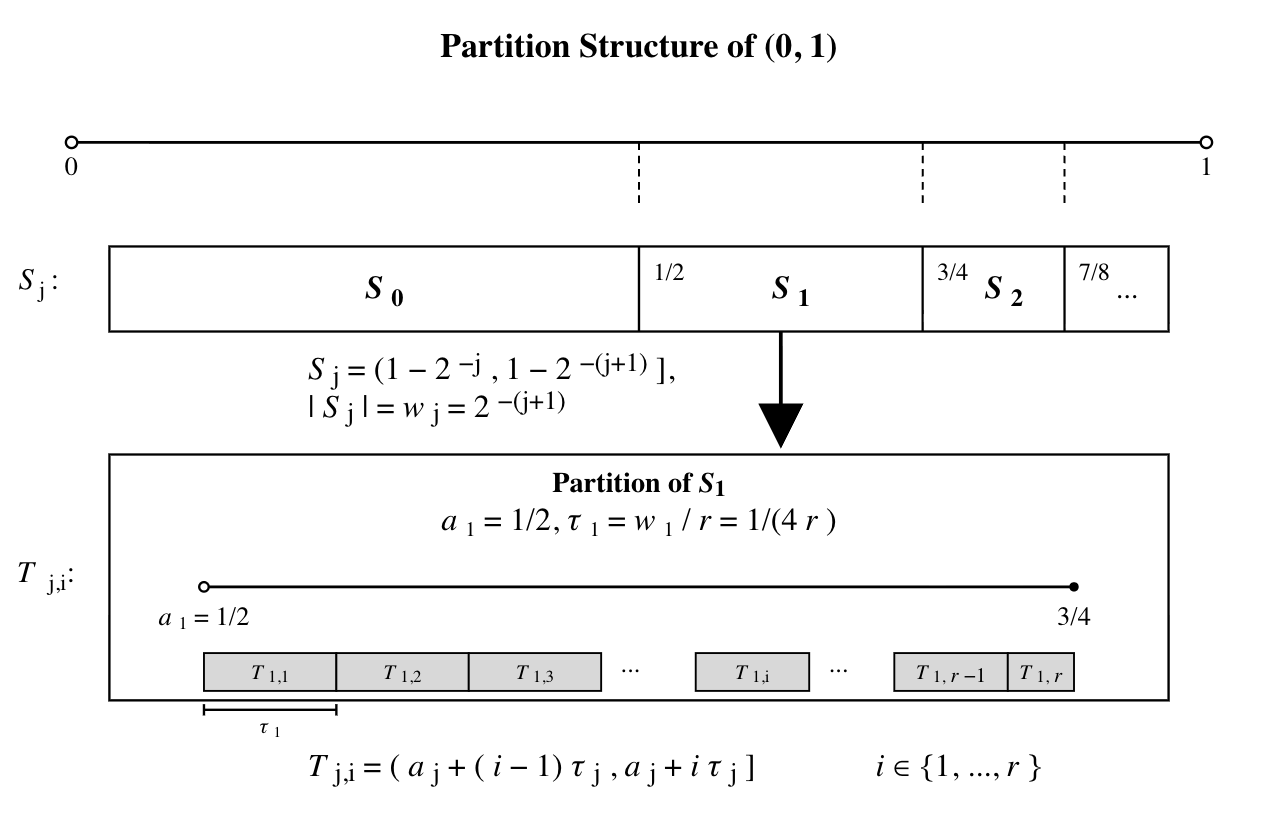}
  \caption{The partition structure of $(0, 1)$ into bands $S_j$ and tiles $T_{j,i}$}
  \label{fig:example}
\end{figure}

\emph{Define the coloring} on $\N^k$ by
\[
\mathrm{Color}(\mathbf{a})=i\in\{1,\dots,r\}
\quad\Longleftrightarrow\quad
\varphi(\mathbf{a})\in T_{j,i}\ \text{for the unique } j\text{ with }\varphi(\mathbf{a})\in S_j.
\]

\smallskip

\subsection{Preliminary lemmas}
\begin{lemma}\label{lm2.1}
Let $\bigcup_{i=1}^tF_i$ be any partition (coloring) of $\N^k$. There exist an $F_m$ and an interval $(x,y)\subseteq [0,1)$ such that  $\varphi(F_m)$ is everywhere dense in $(x, y)$. 
  
\end{lemma}
\begin{proof}
    Assume that the opposite holds. Let $(x_0,y_0)\subseteq [0,1)$ be any interval.  Inductively,
let $j\in\{1,2,\ldots,t\}$. Assume that $\{\varphi(\mathbf{x})\colon \mathbf{x}\in
F_j\}$ is not dense in $(x_{j-1},y_{j-1})$, hence there is an interval $(x_j, y_j)$ with  $(x_j,y_j)\subseteq (x_{j-1},y_{j-1})$ and $\{ \varphi(\mathbf{x})\colon x\in F_j\}\cap(x_j,y_j)=\emptyset$. After $t$ many steps, we get that $(x_t,y_t)\bigcap\bigcup_{j=1}^t\{\varphi(\mathbf{x})\colon \mathbf{x}\in F_j\} =\emptyset$, so
$(x_t,y_t)\cap\{\varphi(\mathbf{x})\colon \mathbf{x}\in \N^k\}=\emptyset$. This leads to a contradiction.
\end{proof}

For $x\in\mathbb R$, define
\[
\nnorm{x}:=\min_{n\in\mathbb Z}|x-n|\in[0,\tfrac12].
\]
Equivalently, with the fractional part $\fr{x}\in[0,1)$,
\[
\nnorm{x}=\min\{\fr{x},\,1-\fr{x}\}.
\]

We now recall some basic identities and bounds.
For all $x,y\in\R$ and $n\in\Z$:
\begin{enumerate}
\item \emph{Range and zeros:} $0\le \nnorm{x}\le \tfrac12$, and $\nnorm{x}=0 \iff x\in\Z$.
\item \emph{Periodicity and parity:} $\nnorm{x+n}=\nnorm{x}$ and $\nnorm{-x}=\nnorm{x}$.
\item \emph{Triangle inequality:} $\nnorm{x+y}\le \nnorm{x}+\nnorm{y}$.
\item \emph{Lipschitz property:} $\big|\nnorm{x}-\nnorm{y}\big|\le \nnorm{x-y}$.
\item \emph{Difference via fractional parts:} $\nnorm{\fr{a}-\fr{b}}=\nnorm{a-b}$ for all $a,b\in\R$.
\end{enumerate}

\begin{lemma}\label{lm2.2}
Let $\beta \notin \Q$. For every $\varepsilon>0$ there exists a sequence $X=\{x_n\}_{n\ge 1}$ in $\mathbb{N}$ such that for every $h\in \mathrm{FS}(X)$, we have $\|\beta h\|<\varepsilon$.
\end{lemma}
\begin{proof}
There is an integer $x_1\in \N$ such that $\|\beta x_1\|<\varepsilon/2$. Assume that the sequence $X_{k}=\{x_1,x_2,\dots ,x_{k}\}$ has been defined in such a way that for every $h\in \mathrm{FS}(X_k)$,  $\|\beta h\|<\varepsilon\left(1-\frac{1}{2^k}\right)$ (this statement is true, e.g., for $k=1$). Now choose an $x_{k+1}\in \N$ for which $\|\beta x_{k+1}\|<\varepsilon/2^{k+1}$. We have
$$
\mathrm{FS}(X_{k+1})=\mathrm{FS}(X_{k})+\{0,x_{k+1}\}.
$$
Thus, if $h\in \mathrm{FS}(X_{k+1})$ then $h=h'+p$, where $p\in \{0,x_{k+1}\}$ and $h'\in \mathrm{FS}(X_{k})$. Hence, \[\|\beta h\|= \| \beta(h'+p)\|\leq \|\beta h'\|+\|\beta p\|<\varepsilon\left(1-\frac{1}{2^k}\right)+\frac{\varepsilon}{2^{k+1}}=\varepsilon\left(1-\frac{1}{2^{k+1}}\right).\]
This completes the proof.
\end{proof}

Let $A\subset\T=\R/\Z$ and $\varepsilon>0$. Note that 
\[
\mathbb T := \mathbb R/\mathbb Z \cong [0,1)\quad\text{(addition mod $1$).}
\]
Define
\[
A_{+\varepsilon}\ :=\ \{\,x\in\T:\ \dist(x,A)<\varepsilon\,\}
\quad\text{and}\quad
A_{-\varepsilon}\ :=\ \{\,x\in A:\ \dist(x,\partial A)>\varepsilon\,\},
\]
where by $\partial A$ we mean the boundary of $A$. 

If $A=(a,b)\subset\T$ is an interval (interpreting endpoints mod $1$), then
\[
(a,b)_{+\varepsilon}=(a-\varepsilon,\ b+\varepsilon),
\qquad
(a,b)_{-\varepsilon}=(a+\varepsilon,\ b-\varepsilon)\ \ (\text{when }b-a>2\varepsilon).
\]

\begin{lemma}\label{lm2.3}
Let $A\subset\T$ be an interval and $\sigma,\gamma\in\T$ with $\nnorm{\gamma-\sigma}<\varepsilon$. Then
\[
(A+\sigma)_{-\varepsilon}\ \subset\ A+\gamma\ \subset\ (A+\sigma)_{+\varepsilon}.
\]
\end{lemma}

\begin{proof}
We first observe that
\[
\dist(u,\partial A)\;=\;\dist(u+\sigma,\partial(A+\sigma))\quad\text{for all }u\in\T.
\]

Take $x\in(A+\sigma)_{-\varepsilon}$. Then $\dist(x,\partial(A+\sigma))>\varepsilon$.
Write $x=u+\sigma$ with $u\in A$. So,
\[
\dist(u,\partial A)=\dist(u+\sigma,\partial(A+\sigma))=\dist(x,\partial(A+\sigma))>\varepsilon.
\]
Since $\nnorm{\gamma-\sigma}<\varepsilon$, the circle point $u+(\sigma-\gamma)$ lies within distance $<\varepsilon$ of $u$,
so $u+(\sigma-\gamma)\in A$. Therefore,
\[
x=(u+(\sigma-\gamma))+\gamma\in A+\gamma.
\]

Take $y\in A+\gamma$, so $y=v+\gamma$ with $v\in A$. From $\nnorm{\gamma-\sigma}<\varepsilon$, we have $v+(\gamma-\sigma)\in A_{+\varepsilon}$. Thus,
\[
y=(v+(\gamma-\sigma))+\sigma\in (A+\sigma)_{+\varepsilon}.
\]
This completes the proof.
\end{proof}

\subsection{Proof of Theorem \ref{HD}}
For $1\le i\le r$, let $E_i$ be the set of $\mathbf{x}\in \mathbb{N}^k$ such that $\mathrm{Color}(\mathbf{x})=i$.

For every finite coloring of $\mathbb{N}^k$ with $t$ colors, $\mathbb N^k=\bigcup_{m=1}^t F_m$. By Lemma~\ref{lm2.1}, choose $m$ and an interval
$J=(x,y)\subset(0,1)$ with $\varphi(F_m)$ dense in $J$. Write $\delta:=y-x>0$.

Choose a band $S_{j^\ast}$ with width $|S_{j^\ast}|=2^{-({j^\ast+1)}}\le \delta/3$. Choose $x_0\in\mathbb N$ so that
\[
J+\sigma\ \supset\ S_{j^\ast}\quad\text{mod }1,\qquad \sigma:=\{x_0\beta\}.
\]
Now choose $\varepsilon>0$ so small that
\[
S_{j^\ast}\ \subset\ (J+\sigma)_{-\varepsilon}=\{u\in J+\sigma:\ \operatorname{dist}(u,\partial(J+\sigma))>\varepsilon\}.
\]
Apply Lemma~\ref{lm2.2} to obtain $\{x_n\}_{n\ge 1}$ with $\|\beta h'\|<\varepsilon$ for all $h'\in \mathrm{FS}(\{x_n\}_{n\ge 1})$.
For any $h=x_0+h'$ (so $h\in x_0+\mathrm{FS}(\{x_n\}_{n\ge 1})$), set $\gamma:=\{h\beta\}$. Then
\[
\|\gamma-\sigma\|=\|\{h\beta\}-\{x_0\beta\}\|=\|h'\beta\|<\varepsilon.
\]
By Lemma \ref{lm2.3}, one has
\[
(J+\sigma)_{-\varepsilon}\ \subset\ J+\gamma\ \subset\ (J+\sigma)_{+\varepsilon},
\]
hence,
\[
J+\gamma\ \supset\ (J+\sigma)_{-\varepsilon}\ \supset\ S_{j^\ast}.
\]
Because $\varphi(\mathbf{x}+h\mathbf{1}_k)=\varphi(\mathbf{x})+\gamma$, we have that $\varphi(F_m+h\mathbf{1}_k)$ is dense in $J+\gamma$,
and in particular intersects each tile $T_{j^*, i}$ inside $S_{j^\ast}$ infinitely often.
By the definition of $E_i$ this gives
\[
(F_m+h\mathbf{1}_k)\cap E_i\ \text{ is infinite for every } i=1,\dots,r.
\]
This holds for every $h\in x_0+\mathrm{FS}(\{x_n\}_{n\ge 1})$, completing the proof.

\section{Polynomial shifts -- Proof of Theorem \ref{thmpolynomial-main}}

We proceed along the same lines as in the proof of Theorem~\ref{HD}, except that we need to modify Lemma~\ref{lm2.2}, as follows. 
\begin{theorem}\label{lm2.2'}
Let $\beta \notin \Q$. Let $P^{(1)},\dots,P^{(f)}\in\Z[x]$ be non-constant polynomials satisfying $P^{(1)}(0)=\cdots=P^{(f)}(0)=0$. For every $\varepsilon>0$, there exists a set $X_{\varepsilon}\subset \mathbb{N}$ of positive lower density such that for each fixed $h\in X_{\varepsilon}$,  for every $i\in \{1,\dots,f\}$, we have $\nnorm{\{\beta P^{(i)}(h)\}}<\varepsilon$. 
\end{theorem}

The case $f=1$ was first proved by Weyl from 1910s.
\begin{lemma}[Weyl \cite{Hweyl1, Hweyl2}]
Let $\beta \notin \Q$ and $P(x) \in \mathbb{Z}[x]$. Let $J \subseteq (0,1)$ be a finite union of intervals. Then the asymptotic density of the set
\[
H = \{n \in \mathbb{N} : \{\beta P(n)\} \in J\}
\]
exists and $d(H) = \mu(J)$. Here $\mu$ is the Lebesgue measure.
\end{lemma}
For $f \geq 2$, a weaker version of Theorem~\ref{lm2.2'} can be deduced from the works of Schmidt \cite{schmidt} and Maynard \cite{JM}. More precisely, they proved that 
for every $x\ge1$, there is $h<x$ satisfying the inequality \;
$\max_i\big\|\{\beta P^{(i)}(h)\}\big\|\ \le\ C\,x^{-\theta},$
where the constants $C$ and $\theta$ depend on the polynomials $P^{(i)}$. This implies that there exists \emph{at least one} $h \in \mathbb{N}$ such that
$\left\| \{\beta P^{(i)}(h)\} \right\| < \varepsilon$ for every $i$.
If, in addition, we also assume that the polynomials $P^{(i)}$ have no common positive integer root, then it is not hard to deduce that there are \emph{infinitely many} values of $h$ meeting the requirements. Our Theorem~\ref{lm2.2'} guarantees the existence of a set of such values $h$ with \emph{positive lower density}, without any additional assumption.
\smallskip

Taking the above modification into account, 
we outline the proof of Theorem \ref{thmpolynomial-main}, for com\-plete\-ness.
\begin{proof}[Proof of Theorem \ref{thmpolynomial-main}]
By Lemma~\ref{lm2.1}, choose $m$ and an interval $J=(x,y)\subset(0,1)$ such that $\varphi(F_m)$ is dense in $J$.
Write $\delta:=y-x>0$.

Choose a band $S_{j^\ast}$ with width $|S_{j^\ast}|=2^{-(j^\ast+1)}\le \delta/3$.
Pick $x_0\in\mathbb N$ and put $\sigma:=\{x_0\beta\}$, so that
\[
  J+\sigma\ \supset\ S_{j^\ast}\quad\text{mod }1.
\]
Choose $\varepsilon>0$ so small that
\[
  S_{j^\ast}\ \subset\ (J+\sigma)_{-\varepsilon}
  :=\{u\in J+\sigma:\ \dist(u,\partial(J+\sigma))>\varepsilon\}.
\]

It follows from Theorem \ref{lm2.2'} that there exists a set $X_{\varepsilon}\subset \mathbb{N}$ of positive lower density such that for any $h\in X_{\varepsilon}$ one has
\[
  \|\{\beta P^{(j)}(h)\}\|<\varepsilon \quad\text{for all } j=1,\dots,f.
\]

Moreover, since the leading coefficient of each $P^{(j)}$ is positive,
we may choose $h$ large enough to make sure that $P^{(j)}(h)\ge 0$ for all $j$. The set of $h$ satisfying this condition still has positive lower density.

For each fixed $j$, set
\[
  \gamma_j':=\sigma+ \{\beta P^{(j)}(h)\}\mod 1,
\]
Then $\|\gamma_j'-\sigma\|=\|\{\beta P^{(j)}(h)\}\|<\varepsilon$, and by Lemma~\ref{lm2.3},
\[
  (J+\sigma)_{-\varepsilon}\ \subset\ J+\gamma_j'\ \subset\ (J+\sigma)_{+\varepsilon}.
\]
In particular,
\[
  J+\gamma_j'\ \supset\ (J+\sigma)_{-\varepsilon}\ \supset\ S_{j^\ast}.
\]

Since $P^{(j)}(h)\ge 0$, the translate $F_m+(x_0+P^{(j)}(h))\mathbf{1}_k$ lies in $\mathbb N^k$, and
\[
  \varphi\!\big(\mathbf{x}+(x_0+P^{(j)}(h))\mathbf{1}_k\big)
  \;=\;\varphi(\mathbf{x})+\gamma_j' \mod 1.
\]
Because $\varphi(F_m)$ is dense in $J$, it follows that $\varphi(F_m+(x_0+P^{(j)}(h))\mathbf{1}_k)$
is dense in $J+\gamma_j'$, hence intersects each tile $T_{j^\ast,i}\subset S_{j^\ast}$ infinitely many times.
By the definition of the color classes $E_i$, we conclude that
\[
  \big(F_m+(x_0+P^{(j)}(h))\mathbf{1}_k\big)\cap E_i
  \quad\text{is infinite for every } i=1,\dots,r.
\]
This holds for every $j=1,\dots,f$ at the same $h$.
Since there are infinitely many such $h$,
the theorem follows. \end{proof}
\subsection{Proof of Theorem \ref{lm2.2'}}
We first recall  Weyl’s equidistribution criteria in \cite{book3}.

\begin{itemize}
    \item One-dimensional torus $\T$: A sequence $\{x_n\}_{n\ge 1}$, with $x_n\in \T$, is equidistributed iff
\[
 \lim_{N\to\infty} \frac{1}{N}\sum_{n=1}^{N} e^{2\pi i k x_n}\;=\;0
  \qquad\text{for every } k\in \Z\setminus\{0\}. \quad\quad (\cite[\mbox{Theorem 2.1, Chapter 1}]{book3}) 
\]

\item Multidimensional torus $\T^f$: A sequence $\{\mathbf x_n\}_{n\ge 1}$, with $\mathbf{x}_n\in \T^f$, is equidistributed iff
\[
  \lim_{N\to\infty}\frac{1}{N}\sum_{n=1}^{N} e^{2\pi i\, \mathbf k\cdot \mathbf x_n}\;=\;0
  \qquad\text{for every } \mathbf k\in \Z^f\setminus\{\mathbf 0\}.\quad\quad (\cite[\mbox{Theorem 6.2, Chapter 1}]{book3})
\]

\item For a closed subgroup $H\le \T^f$ (subtorus version): Let $\{\mathbf x_n\}_{n\ge 1}\subset H$. Then $\{\mathbf x_n\}_{n\ge 1}$ is equidistributed \emph{on $H$} iff for every
non-trivial character $\chi$ of $H$,
\[
  \lim_{N\to\infty}\frac{1}{N}\sum_{n=1}^{N} \chi(\mathbf x_n)\;=\;0. \quad\quad (\cite[\mbox{Corollary 1.2, Chapter 4}]{book3})
\]
Note that every character of $H$ is of the form
\[
  \chi_{\mathbf m}(\mathbf z) \;=\; e^{2\pi i\, \mathbf m\cdot \mathbf z},
  \qquad \mathbf m\in \Z^f. \quad\quad (\cite[\mbox{Appendix C.3}]{EW})
\]
By non-trivial on $H$, we mean that $\mathbf m$ is not in the annihilator 
$$\Ann(H):=\; \big\{\,\mathbf m\in \Z^f \;:\; \chi_{\mathbf m}\!\restriction_{H}\equiv 1 \,\big\}.$$ So the criterion becomes
\begin{equation}\label{WEC}
  \lim_{N\to \infty}\frac{1}{N}\sum_{n=1}^{N} e^{2\pi i\, \mathbf m\cdot \mathbf x_n}\;=\;0
  \qquad\text{for all } \mathbf m\in \Z^f\setminus \Ann(H).
\end{equation}
\end{itemize}

Let $P^{(1)},\dots,P^{(f)}\in\Z[x]$ with $P^{(1)}(0)=\cdots=P^{(f)}(0)=0$, and let $\beta\in\R\setminus\Q$.
Define
\[
\mathbf v(n):=\big(\{\beta P^{(1)}(n)\},\dots,\{\beta P^{(f)}(n)\}\big)\in\T^f,
\qquad n\in\N,
\]
the \textit{relation lattice}
\[
R:=\Big\{\,\mathbf{m}=(m_1,\dots,m_f)\in\Z^f:\ Q_{\mathbf{m}}:=\sum_{i=1}^f m_i P^{(i)} \equiv 0 \text{ in }\Z[x]\,\Big\},
\]
the relation-defined subtorus
\[
H:=\{\,\mathbf{z}\in\T^f:\ e^{2\pi i\, \mathbf{m}\cdot \mathbf{z}}=1\ \text{ for all }\mathbf{m}\in R\,\},
\]
and the closed subgroup generated by the sequence $\{\mathbf v(n)\}_{n\ge 1}$
\[
H':=\overline{\ \langle\,\mathbf v(n):n\in\N\,\rangle\ } \ \le\ \T^f
\quad(\text{closure of all finite $\Z$-linear combinations}).
\]
Note that closure of a subgroup is a subgroup.

The proof of Theorem \ref{lm2.2'} makes use of the following lemmas.
\begin{lemma}\label{lem:H=Hprime}
We have \(H=H'\), that is, the relation-defined subtorus $H$ equals the closed subgroup generated by the sequence $\{\mathbf v(n)\}_{n\ge 1}$.
\end{lemma}
\begin{lemma}\label{prop:weyl-H}
The sequence $\{\mathbf v(n)\}_{n\ge 1}$ is equidistributed on $H$ with respect to Haar measure $m_H$.
\end{lemma}

We now prove Theorem \ref{lm2.2'}.
\begin{proof}[Proof of Theorem \ref{lm2.2'}]
    Consider the open box
\[
U:=(-\varepsilon,\varepsilon)^f\subset\T^f.
\]
Since $\mathbf{0}\in H$ and $U$ is a neighborhood of $\mathbf{0}$, we have $m_H(U\cap H)>0$ (see \cite[page 341]{book0}).

By Lemma~\ref{prop:weyl-H}, there exist infinitely many $h\in\N$ with
\[
\mathbf v(h)=\big(\{\beta P^{(1)}(h)\},\dots,\{\beta P^{(f)}(h)\}\big)\in U\cap H.
\]
To see this, we argue as follows. 

We know that $H$ is a closed subgroup of $\mathbb{T}^f$ (from Lemma \ref{lem:H=Hprime} or its definition), so it is locally compact. Thus, the measure $m_H$ is regular (for this fact, see, for example, \cite[Chapter 2]{book0} or \cite[Chapter 44]{fre}). Assume that $m_H(U\cap H)=c>0$, then we can find a compact set $K\subset U\cap H$ such that $m_H(K)\ge c/2$. Notice that $H$ is compact Hausdorff, and therefore normal, so one can apply Urysohn's lemma (see, for example, \cite[Theorem 33.1]{munkres2000}) to find a continuous function $F\colon H\to [0, 1]$ such that $
\mathbf{1}_K \le F \le \mathbf{1}_{U\cap H}.$
This implies
\[
\int_H F\,dm_H \ge \int_H \mathbf{1}_K\,dm_H = m_H(K) \ge c/2.
\]
On the other hand, from the equidistribution of $\mathbf{v}(h)$ in $H$, we know that
\[
\lim_{N\to \infty}\frac{1}{N}\sum_{h=1}^N F(\mathbf{v}(h)) =\int_H F\,dm_H \ge c/2.
\]
Since $\mathbf{1}_{U\cap H}(\mathbf{v}(h)) \ge F(\mathbf{v}(h))$ for all $h$, we have:
\[
\frac{1}{N}\sum_{h=1}^N \mathbf{1}_{U\cap H}(\mathbf{v}(h)) \ge \frac{1}{N}\sum_{h=1}^N F(\mathbf{v}(h)).
\]
So
\[
\liminf_{N\to\infty} \frac{1}{N}\,\#\{h \le N: \mathbf{v}(h) \in U\cap H\} \ge c/2 > 0.
\]
This implies the existence of the set $X_\varepsilon\subset \mathbb{N}$ with positive lower density such that for each fixed $h\in X_\varepsilon$,  for every $i\in \{1,\dots,f\}$, we have $\nnorm{\{\beta P^{(i)}(h)\}}<\varepsilon$. Note that the density of $X_\varepsilon$ depends on $\varepsilon$ and polynomials.
\end{proof}
\subsection{Proofs of Lemmas \ref{lem:H=Hprime} and \ref{prop:weyl-H}}
We first recall some results from \cite{book0, book3}.

The following is known as the 1-dimensional form of Weyl's polynomial equidistribution criterion.
\begin{theorem}[Theorem 3.2, Chapter 1, \cite{book3}]\label{pro31}
    If \(Q(x)=c_d x^d+\cdots+c_1 x+c_0\in\mathbb R[x]\) is a non-constant polynomial
and at least one coefficient \(c_j\), $j>0$, is irrational, then
\[
  \lim_{N\to\infty}\frac{1}{N}\sum_{n=1}^{N} e^{2\pi i\, Q(n)} \;=\; 0 .
\]
Equivalently, the sequence $\{Q(n)\}_{n\ge 1}$ is uniformly distributed mod \(1\).
\end{theorem}
Since $\T^f$ is a locally compact abelian group, the next proposition follows by Pontryagin duality.
\begin{proposition}[Proposition 4.39, \cite{book0}]\label{lem:annihilator-ker}
Let $K\le \T^f$ be a closed subgroup and
\[
\Ann(K):=\{\,\mathbf{m}\in\Z^f:\ e^{2\pi i\, \mathbf{m}\cdot \mathbf{z}}=1\ \text{ for all }\mathbf{z}\in K\,\}.
\]
Then
\begin{equation}\label{eq:annihilator-ker}
K \;=\; \big\{\, \mathbf{z}\in\T^f:\ e^{2\pi i\, \mathbf{m}\cdot \mathbf{z}}=1\ \text{ for all } \mathbf{m}\in \Ann(K)\,\big\}.
\end{equation}
\end{proposition}
To complete the proof of Lemma \ref{prop:weyl-H}, we also need to establish that $\Ann(H)=R$. This will be done in two steps: first show that $\Ann(H')=R$, then we prove that $H=H'$ (Lemma \ref{lem:H=Hprime}).
\begin{proposition}\label{lm3.11}
The annihilator of $H'$ in $\Z^f$ equals $R$, i.e.
\[
\Ann(H')=\{\,\mathbf{m}\in\Z^f:\ e^{2\pi i\,\mathbf{m}\cdot \mathbf{z}}=1 \ \text{for all }\mathbf{z}\in H'\,\}=R.
\]
\end{proposition}
\begin{proof}
    The characters of $\T^f$ are $\chi_\mathbf{m}(z)=e^{2\pi i\,\mathbf{m}\cdot \mathbf{z}}$ with $\mathbf{m}\in\Z^f$.
By definition,
\[
\Ann(H')=\{\,\mathbf{m}\in\Z^f:\ \chi_\mathbf{m}(\mathbf{z})=1\text{ for all }\mathbf{z}\in H'\,\}.
\]

We first prove that $R\subseteq \Ann(H')$.
If $\mathbf{m}\in R$, then $Q_\mathbf{m}\equiv 0$. Hence, for every $n$,
\[
\chi_\mathbf{m}(\mathbf v(n))=e^{2\pi i\,\mathbf{m}\cdot \mathbf v(n)}
=e^{2\pi i\,\beta\,Q_\mathbf{m}(n)}=e^{2\pi i\cdot 0}=1.
\]
So, $\chi_\mathbf{m}$ equals $1$ on $\{\mathbf v(n)\}_{n\ge 1}$ and, therefore, on the subgroup they generate,
and by continuity also on its closure $H'$. Thus, $\mathbf{m}\in\Ann(H')$.

We now show that $\Ann(H')\subseteq R$.
Let $\mathbf{m}\in\Ann(H')$. In particular, we have $\chi_\mathbf{m}(\mathbf v(n))=1$ for all $n\in\N$, i.e.
\[
e^{2\pi i\,\beta\,Q_\mathbf{m}(n)}=1\quad\text{for all }n\in\N
\quad\Longrightarrow\quad \beta\,Q_\mathbf{m}(n)\in\Z\ \text{ for all }n\in\N.
\]
Suppose, by contradiction, that $Q_\mathbf{m}\not\equiv 0$. Then there exists $n_0\in\N$
with $Q_\mathbf{m}(n_0)\neq 0$. From $\beta\,Q_\mathbf{m}(n_0)\in\Z$ we get
\(
\beta = \dfrac{u}{Q_\mathbf{m}(n_0)}
\)
for some $u\in\Z$, which makes $\beta$ rational, which leads to a contradiction since
$\beta\notin\Q$. Hence, $Q_\mathbf{m}\equiv 0$, i.e. $\mathbf{m}\in R$.

In other words, $\Ann(H')=R$.
\end{proof}

\begin{proof}[Proof of Lemma \ref{lem:H=Hprime}]
    From Proposition \ref{lm3.11}, we know that $\Ann(H')=R$.

For any closed subgroup $K\le \T^f$, Proposition \ref{lem:annihilator-ker} implies
\begin{equation}\label{hpp}
K \;=\; \{\, \mathbf{z}\in\T^f:\ e^{2\pi i\, \mathbf{m}\cdot \mathbf{z}}=1\ \text{ for all } \mathbf{m}\in \Ann(K)\,\}.
\end{equation}
Applying (\ref{hpp}) to $K=H'$ and using $\Ann(H')=R$ gives
\[
H' \;=\; \{\, \mathbf{z}\in\T^f:\ e^{2\pi i\, \mathbf{m}\cdot \mathbf{z}}=1\ \text{ for all } \mathbf{m}\in R\,\} \;=\; H.
\]
This completes the proof.
\end{proof}
\begin{proof}[Proof of Lemma \ref{prop:weyl-H}]
Let $\widehat H$ be the character group of $H$. Then, we know from Proposition \ref{lm3.11} and Lemma \ref{lem:H=Hprime} that $\Ann(H)=R$. Any non-trivial in $\widehat H$ is of the form
$\chi_\mathbf{m}(z)=e^{2\pi i\, \mathbf{m}\cdot \mathbf{z}}$ with some $\mathbf{m}\in\Z^f\setminus R$.
Recall
\[
Q_\mathbf{m}(n)=\sum_{i=1}^f m_i P^{(i)}(n).
\]
With $\mathbf{m}\notin R$ and $P^{(i)}$ are non-constant  with $P^{(i)}(0)=0$, the integer-coefficient polynomial $Q_\mathbf{m}$ is non-constant and $Q_\mathbf{m}(0)=0$.
Since $\beta$ is irrational, at least one coefficient of $\beta Q_\mathbf{m}$ is irrational. By Theorem \ref{pro31},
\[
\lim_{N\to\infty}\frac1N\sum_{n=1}^N e^{2\pi i\,\beta\, Q_\mathbf{m}(n)} =0.
\]
This is equivalent to 
\[\lim_{N\to\infty}\frac1N\sum_{n=1}^N e^{2\pi i\,\mathbf{m}\cdot \mathbf{v}(n)} =0        \ \text{ for all } \mathbf{m}\in \mathbb{Z}^f\setminus\Ann(H).\]
By Weyl's criterion (\ref{WEC}) on $H$, $\{\mathbf v(n)\}_{n\ge 1}$ is equidistributed in $H$.
\end{proof}

\section{Finite abelian groups -- Proof of Theorem \ref{thm3.1}}
Since every finite cyclic group is isomorphic to $\mathbb{Z}_N$ for some $N$, we first prove the theorem for $G = \mathbb{Z}_N$.

\begin{proof}[Proof of Theorem \ref{thm3.1} for cyclic groups]
 Choose $k := 1 + 2^{r+3} t$. Set $\Delta_1 :=k^{r-1} \left\lfloor \frac{N-1}{k^{r-1}S_k} \right\rfloor$, where $S_k = 1 + \frac{1}{k} + \cdots + \frac{1}{k^{r-1}}$. Then $k^{r-1} \mid \Delta_1$ and $\Delta_1 S_k \leq N - 1$.
      
First, we define $E_1$ as $E_1=[0,\Delta_1]\subseteq \mathbb{Z}_N$. For $1<i\leq r$, let $\Delta_i=\Delta_{i-1}/k$, $E_i=(u_i,v_i]$, where $u_i=\sum_{j=1}^{i-1}\Delta_j$ and $v_i=\sum_{j=1}^{i}\Delta_j$. Note that if $\mathbb{Z}_N\setminus (E_1\cup \cdots \cup E_r)=R$, then $|R|<k^{r-1}S_k$. We can add $R$ to any set $E_i$, say, to $E_r$, to extend it to a complete covering of $\mathbb{Z}_N$, and will not use it in the next steps. Notice that $E_1, \ldots, E_r$ are intervals.

Let $\chi$ be any $t-$coloring of $\mathbb{Z}_N$ and let $F_m$ be one of the largest parts. Then we have $|F_m|\geq N/t=:\beta \cdot N$. 

Let $h_1$ be a translation such that $|(F_m+h_1)\cap E_1|\geq  \beta|E_1|$. It is not hard to see that such an $h_1$ exists. Indeed, denote by $E_1(x)$ the indicator of the interval $E_1$ and $F_m(x)$ the indicator function of $F_m$. We have
\[
\frac{1}{N}\sum_{h\in \mathbb{Z}_N}\sum_{x\in \mathbb{Z}_N}E_1(x)F_m(x+h)=\frac{1}{N}\sum_{x\in \mathbb{Z}_N}E_1(x)\sum_{h\in \mathbb{Z}_N}F_m(x+h)= \frac{|F_m|}{N}\sum_{x\in \mathbb{Z}_N}E_1(x)\ge \beta|E_1|.
\]
The left-hand side is the average of $|(F_m+h)\cap E_1|$, hence there is an $h_1$, for which $|(F_m+h_1)\cap E_1|\geq  \beta|E_1|$.

In the next step, we need to define $h_2, \ldots, h_r$ such that $h:=\sum_{i=1}^rh_i$ satisfies the desired property. We present a complete argument for $h_2$ and $E_2$. For $i\ge 3$, we repeat the same process.

We divide $E_1$ into $k$ equal intervals, say $I_1, \ldots, I_k$, and note that $|E_1\cap (F_m+h_1)|\ge \beta |E_1|$. There exists an interval $I_j=(a_j, b_j]\subset E_1$ with $|I_j|=\frac{\Delta_1}{k}$ such that $|I_j\cap (F_m+h_1)|\ge \frac{\beta}{2}|I_j|$. Assume $I_j$ is the right most interval with this property. Since $I_j$ is the right most interval, the intervals $I_{j+1}, \ldots, I_k$ together contain at most $\beta \Delta_1/2$ elements. 

Choose $h_2=u_2-a_j$, then it is clear that $h_2\le \Delta_1$, $I_j+h_2=E_2$, 
and 
\[|E_2\cap (F_m+h_1+h_2)|\ge \frac{\beta}{2}|E_2|.\]

Note that $I_1+h_2, \ldots, I_{j-1}+h_2$ are intervals fully contained in $E_1$, we have
\[|E_1\cap (F_m+h_1+h_2)|\ge \beta \Delta_1-
\sum_{i=j+1}^k|I_{i}\cap (F_m+h_1)|-|I_j\cap (F_m+h_1)|\ge
\frac{\beta\Delta_1}{2}-\frac{\Delta_1}{k}\ge \frac{\beta\Delta_1}{4}\] when $k>\frac{4}{\beta}$. 

We repeat this process for $E_3,\ldots, E_r$. As a consequence, we obtain $h_3, \ldots, h_r$. Let $h:=\sum_{i=1}^rh_i$. At the end, we have $|E_s\cap (F_m+\sum_{i=1}^{s+1} h_i)|\ge \frac{\beta}{2^{s+1}}|E_s|$ for all $1\le s\le r-1$, and $|E_s\cap (F_m+\sum_{i=1}^{s} h_i)|\ge \frac{\beta}{2^{s}}|E_s|$ for all $1\le s\le r$.
Using the facts that $k= 1 + 2^{r+3} t$ and $h_i\le \Delta_{i-1}$ for all $i\ge 2$, we have that for each $1\le s\le r-2$,
\[\sum_{i=s+2}^rh_i\le \Delta_{s+1}+\cdots+\Delta_{r-1}\le \Delta_{s+1}\left(1+\frac{1}{k}+\cdots+\frac{1}{k^{r-s}}\right)\le \frac{k\Delta_{s+1}}{k-1}\le \frac{\beta\Delta_s}{2^{s+3}}\]
holds. Therefore,
\[|E_s\cap (F_m+h)|\ge|E_s\cap (F_m+h_1+\cdots+h_{s+1})|-\sum_{i=s+2}^rh_i \ge |E_s\cap (F_m+h_1+\cdots+h_{s+1})|-\frac{\beta\Delta_s}{2^{s+3}}.\]
Hence,
\[|E_s\cap (F_m+h)|\ge  \frac{\beta\Delta_s}{2^{s+1}}-\frac{\beta\Delta_s}{2^{s+3}}\ge \frac{\beta \Delta_s}{2^{s+2}}.\]
This is equal to\[|G|\cdot \frac{\Delta_1}{tk^{s-1}2^{s+2}|G|}.\]
Set $\alpha=\frac{\Delta_1}{tk^{r-1}2^{r+2}|G|}$. Note that we need $|G|\ge 1+k^{r-1}|S_k|$ to guarantee that $\Delta_r\ge 1$. A direct computation shows that 
\[\alpha \ge \frac{2}{(1+t2^{r+3})^r-1}.\]
This completes the proof.
\end{proof}
\begin{proof}[Proof of Theorem \ref{thm3.1} for general abelian groups]
Assume without loss of generality that $G=\mathbb{Z}_{N}\times G'$, where $G'$ is a finite abelian group and $N\ge N_0$. We denote the zero element in $G'$ by $0_{G'}$.
Fix $r,t\in\mathbb{N}$. Let $E_1,\dots,E_r\subset\mathbb{Z}_{N}$ be the partition family given by Theorem \ref{thm3.1}.
Define $\widetilde{E}_i := E_i\times G'$ of $G$.

Then, $\bigcup_{i=1}^r \widetilde{E}_i$ forms a partition of $G$, and for every finite coloring of $G$ with $t$ colors, $G=\bigcup_{m=1}^t F_m$, there exist an index $m^\ast\in\{1,\dots,t\}$
and a shift $h\in G$ such that
\[
\big|(F_{m^\ast}+h)\cap \widetilde{E}_i\big| \ \ge\ \frac{\alpha(r,t)}{t}\,|\widetilde{E}_i|
\qquad\text{for all }i=1,\dots,r,
\]
where $\alpha(r,t)>0$ is the same constant as in the first part of the theorem.

We now prove the last statement.

For each $x \in \mathbb{Z}_{N}$ and each $m \in \{1,\dots,t\}$, define 
\[
A_m(x) = \{y \in G' : (x,y) \in F_m\}.
\]
Since $\bigcup_{m=1}^t F_m$ is a partition of $G$, the sets $\{A_m(x)\}_{m=1}^t$ form a partition of $G'$ for each fixed $x$. In particular,
\[
\sum_{m=1}^t |A_m(x)| = |G'| \quad \text{for all } x \in \mathbb{Z}_{N}.
\]
By the pigeonhole principle, for each $x$ there exists at least one index $m_0 \in \{1,\dots,t\}$ such that
\[
|A_{m_0}(x)| \geq \frac{|G'|}{t}.
\]
For each $x \in \mathbb{Z}_{N}$, we choose one such $m_0$ and denote it by $m(x)$. Define
\[
C_m = \{x \in \mathbb{Z}_{N} : m(x) = m\} \quad \text{for } m = 1,\dots,t.
\]
Then $\bigcup_{m=1}^t C_m$ is another partition of $\mathbb{Z}_N$.

Applying Theorem \ref{thm3.1} to $\mathbb{Z}_N$, we obtain an index $m^\ast \in \{1,\dots,t\}$ and an element $h_1 \in \mathbb{Z}_{N}$ such that, for all $i = 1,\dots,r$,
\begin{equation}\label{step1}
|(C_{m^\ast} + h_1) \cap E_i| \geq \alpha(r,t) N.
\end{equation}
Let $h = (h_1, 0_{G'}) \in G$. For a fixed $i \in \{1,\dots,r\}$, we compute
\[
\begin{aligned}
\big|(F_{m^\ast} + h) \cap \widetilde{E}_i\big|
&= \sum_{x \in E_i} \big|\{y \in G' : (x - h_1, y) \in F_{m^\ast}\}\big|.
\end{aligned}
\]
Restricting the sum to those $x$ for which $x - h_1 \in C_{m^\ast}$, gives
\[\big|(F_{m^\ast} + h) \cap \widetilde{E}_i\big|
\geq \sum_{\substack{x \in E_i \\ x - h_1 \in C_{m^\ast}}} \big|\{y \in G' : (x - h_1, y) \in F_{m^\ast}\}\big|.
\]
For $x$ with $x - h_1 \in C_{m^\ast}$, we have, by the definition of $C_{m^\ast}$, that
\[
\big|\{y \in G' : (x - h_1, y) \in F_{m^\ast}\}\big| = |A_{m^\ast}(x - h_1)| \geq \frac{|G'|}{t}.
\]
Therefore,
\[
\big|(F_{m^\ast} + h) \cap \widetilde{E}_i\big| \geq \frac{|G'|}{t} \cdot \big|E_i \cap (C_{m^\ast} + h_1)\big|.
\]
Moreover, we know from (\ref{step1}) that
\[
\big|E_i \cap (C_{m^\ast} + h_1)\big| \geq \alpha(r,t) N,
\]
which implies
\[
\big|(F_{m^\ast} + h) \cap \widetilde{E}_i\big| \geq \frac{|G'|}{t} \cdot \alpha(r,t) N = \frac{\alpha(r,t)}{t} |G|.
\]This completes the proof.
\end{proof}

\section{Special linear groups -- Proof of Theorem \ref{thm3.2}}
To prove this theorem, we make use of the following lemmas. 

\begin{lemma}[Chebyshev's inequality]\label{lm3.22}
Let $X$ be a real-valued random variable with mean
$\mu=\mathbb{E}[X]$ and variance
$\operatorname{Var}(X)=\sigma^{2}$.
For every $t>0$,
\[
\Pr\!\bigl(|X-\mu|\ge t\bigr)\le\frac{\sigma^{2}}{t^{2}}.
\]
If $\theta<\mu$ then
\[
\Pr\bigl[X\le\theta\bigr]
      \le
      \frac{\sigma^{2}}{(\mu-\theta)^{2}}.
\]
\end{lemma}

\begin{lemma}\label{lm5.22}
Fix a constant $t\ge 1$.
Let $U\subseteq\mathbb{F}_{q}^{2}$ satisfy
$|U|=q^{2}/t$.
For every real number $c$ with $0<c<1$ and every prime power
\[
q>\frac{4t}{c},
\]
there exists a nonzero vector $\mathbf{v}\in\mathbb{F}_{q}^{2}$ such that
\[
\#\Bigl\{\lambda\in\mathbb{F}_{q}:
        \bigl|U\cap\ell_{\mathbf{v},\lambda}\bigr|
          \ge\frac{q}{2t}\Bigr\}
     \ge(1-c)q,
\]
where
$
\ell_{\mathbf{v},\lambda}:=\{(x,y)\in\mathbb{F}_{q}^{2}:(x,y)\cdot \mathbf{v}=\lambda\}.
$
\end{lemma}

\begin{proof}
For $\mathbf{v}\in \mathbb{F}_q^2\setminus \{0\}$ and $\lambda\in\mathbb{F}_{q}$ define
$
m(\mathbf{v},\lambda):=\bigl|U\cap\ell_{\mathbf{v},\lambda}\bigr|.
$
Set
$
\mu:=|U|/q=q/t.
$ By a direct computation, we have
\[
\sum_{\mathbf{v}\neq \mathbf{0}}\sum_{\lambda}m(\mathbf{v},\lambda)^{2}
     =(q-1)\bigl(|U|^{2}-|U|\bigr)+(q^{2}-1)|U|
     =(q-1)|U|\bigl(|U|+q\bigr).
\]
There are $q^{2}-1$ nonzero vectors $\mathbf{v}$, so for at least one
$\mathbf{v}^{\ast}$
\[
\sum_{\lambda\in \mathbb{F}_q}m(\mathbf{v}^{\ast},\lambda)^{2}
     \le\frac{(q-1)|U|\bigl(|U|+q\bigr)}{q^{2}-1}
     \le\frac{|U|^{2}}{q}+|U|
     =\frac{q^{3}}{t^{2}}+\frac{q^{2}}{t}.
\]

Slightly abusing the notation, we write $m(\lambda):=m(\mathbf{v}^{\ast},\lambda)$. Let $X$ be a random variable defined by $X(\lambda)=m(\lambda)$, where $\lambda\in \mathbb{F}_q$ is chosen uniformly.  Using the above estimate
\[
\operatorname{Var}(X)
  =\frac{1}{q}\sum_{\lambda}(m(\lambda)-\mu)^{2}
  =\frac{1}{q}\sum_{\lambda}m(\lambda)^{2}-\mu^{2}
  \le\frac{q}{t}.
\]

Set $\theta:=\mu/2=q/(2t)$ and define
$\mathcal{A}:=\{\lambda:m(\lambda)<\theta\}$.
Applying Lemma \ref{lm3.22}, we obtain
\[
\Pr\bigl[X<\theta\bigr]
   \le\frac{\operatorname{Var}(X)}{(\mu-\theta)^{2}}
   =\frac{q/t}{(q/2t)^{2}}
   =\frac{4t}{q}.
\]
Note that $\Pr[X<\theta]=\frac{|\mathcal{A}|}{q}$. So $|\mathcal{A}|\le 4t$.
If $q>4t/c$, then there exists a set $\mathcal{B}$ of at least $(1-c)q$ elements $\lambda$ such that $m(\lambda)\ge q/2t$.

This completes the proof.
\end{proof}
\begin{proof}[Proof of Theorem \ref{thm3.2}]
Let $I_1\cup I_2\cup \cdots \cup I_r$ be a partition of $\mathbb{F}_q$ with $|I_1|=\cdots=|I_{r-1}|=\left\lfloor\frac{q}{r}\right\rfloor$ and $|I_r|=q-(r-1)\left\lfloor\frac{q}{r}\right\rfloor$.

    Let $E_i$ be the set of matrices of the form 
\[\begin{pmatrix}
    x&y\\
    z&w
\end{pmatrix},\]
where $xw-yz=1$, $x\in I_i\subset \mathbb{F}_q$, $y, z\in \mathbb{F}_q$. Then the sets $E_i$ form a partition of $SL_2(\mathbb{F}_q)$. 

For each finite coloring $\bigcup_{i=1}^t F_i$ of $SL_2(\mathbb{F}_q)$ with $t$ colors, let $F_m$ be a set such that $|F_m|\ge q^3/2t$. Assume without loss of generality that the projection of $F_m$ onto the first two coordinates, denoted by $F_m^{1, 2}$, is of size at least $q^2/4t$, and for each $(x, y)\in F_m^{1, 2}$, the number of $z$ is at least $q/4t$.

Applying Lemma \ref{lm5.22} with $c=\frac{1}{2r}$ and $U=F_m^{1,2}$, there exists a nonzero vector $\mathbf{v}=(v_1, v_2)\in\mathbb{F}_{q}^{2}$ such that
\[
\#\Bigl\{\lambda\in\mathbb{F}_{q}:
        \bigl|F_m^{1,2}\cap\ell_{\mathbf{v},\lambda}\bigr|
          \ge\frac{q}{8t}\Bigr\}
     \ge(1-c)q,
\]
where
$
\ell_{\mathbf{v},\lambda}:=\{(x,y)\in\mathbb{F}_{q}^{2}:(x,y)\cdot \mathbf{v}=\lambda\},
$ provided that $q$ is large enough. Let 
\[h=\begin{pmatrix}
    v_1&-\frac{1}{v_2}\\
    v_2&0
\end{pmatrix}.\]
We now show that $|F_m\cdot h\cap E_i|\ge \frac{q^3}{4rt^2}$.
Indeed, since $\#\{(x, y)\cdot \mathbf{v}\colon (x, y)\in F_m^{1, 2}\}\ge (1-c)q$, we can conclude that the set $\mathbf{v}\cdot F_m^{1,2}$ intersects $I_i$ in at least $|I_i|/2$ elements. For each element $\lambda$ in the intersection, we also have that 
\[\#\{(x, y)\in F_m^{1, 2}\colon (x, y)\cdot (v_1, v_2)=\lambda \}\ge \frac{q}{8t}.\]
Furthermore, for each $(x, y)\in F_m^{1,2}$, the number of $z$ such that 
\[\begin{pmatrix}
    x&y\\
    z&w
\end{pmatrix}\in F_m\]
is at least $q/4t$. Hence, for each $\lambda$, there are at least $\frac{q^2}{32t^2}$ matrices $M$ in $F_m$ such that 
\[M\cdot \begin{pmatrix}
    v_1&-\frac{1}{v_2}\\
    v_2&0
\end{pmatrix}=\begin{pmatrix}
    \lambda&*\\
    *&*
\end{pmatrix},~\mbox{where}~~ *\in \mathbb{F}_q.\]
This means that $|F_m\cdot h\cap E_i|\ge |E_i|/2=\frac{q}{2r}\cdot \frac{q^2}{32t^2}=\frac{q^3}{64rt^2}$. 
This completes the proof.
\end{proof}

\section{Acknowledgements}
Norbert Hegyv\'{a}ri was supported by the National Research, Development and Innovation Office NKFIH Grant No K-146387. J\'{a}nos Pach was supported by NKFIH grant K-131529 and ERC Advanced Grant 882971 ``GeoScape". We thank the Vietnam Institute for Advanced Study in Mathematics (VIASM), where some parts of this paper were completed, for its hospitality.


\begin{thebibliography}{00}
\bibitem{Bergelson2} V. Bergelson and B. Weiss, \textit{Translation Properties of Sets of Positive Upper Density}, Proceedings of the American Mathematical Society, \textbf{94}(3) (1985), 371--376.
\bibitem{BL}V. Bergelson and A. Leibman, \textit{Polynomial extensions of Van der Waerden's and Szemer\'{e}di's theorems}, Journal of the American Mathematical Society, \textbf{9} (1996), 725--753.

\bibitem{chapmanI}
J. Chapman and A. Iosevich, \textit{On Rapid Generation of $SL_2(\mathbb{F}_q)$}, Integer, \textbf{9} (2009), A4, 47--52.

\bibitem{ding}
Y. Ding, H. Li, and Z. Zhang, \textit{A modular version of the Brunn-Minkowski inequality and its applications}, arXiv:2509.25025 (2025).
\bibitem{EW}
M. Einsiedler and T. Ward, \textit{Ergodic Theory: with a view towards Number Theory}, Springer Verlag, London, 2011.
\bibitem{Furstenberg1977}
H.~Furstenberg,
\emph{Ergodic behavior of diagonal measures and a theorem of Szemer\'edi on arithmetic progressions},
Journal d'Analyse Math\'ematique, \textbf{31} (1977), 204--256.
\bibitem{book0}
G. B. Folland, \textit{A course in abstract harmonic analysis}, CRC press, 2016.
\bibitem{fre}
D. H. Fremlin, \textit{Measure theory}, Volume 4, Torres Fremlin, 2000.
\bibitem{hart}
D. Hart, A. Iosevich, D. Koh, and M. Rudnev, \textit{Averages over hyperplanes, sum-product theory in vector spaces over finite fields and the Erd\H{o}s-Falconer distance conjecture},  Transactions of the American Mathematical Society, \textbf{363}(6), 3255--3275.
\bibitem{NH} N. Hegyv\'ari, 
\textit{On intersecting properties of partitions of integers},
Combinatorics Probability and Computing, \textbf{14}(3) (2005), 319--323.
\bibitem{HH}
H. Helfgott, \textit{Growth and generation in $SL_2(\mathbb{F}_q)$}, Annals of Mathematics, \textbf{167}(2) (2008), 601--623.
\bibitem{HKR}
F. Hernandez, I. Kousek, and T. Radic, \textit{On density analogs of Hindman’s finite sums theorem}, arXiv:2510.18788 (2025).
 \bibitem{HM} N. Hindman, \textit{Ultrafilters and combinatorial number theory},
Number theory, Carbondale 1979 (Proc. Southern Illinois Conf.,
Southern Illinois Univ., Carbondale, Ill., 1979),  pp. 119--184,
Lecture Notes in Math., 751, Springer, Berlin, 1979.
\bibitem{Hind}
N. Hindman, \textit{Finite sums from sequences within cells of a partition of $\mathbb{N}$}, Journal of Combinatorial Theory, Series A, \textbf{17} (1974), 1--11.


\bibitem{IR}
A. Iosevich and M. Rudnev, \textit{Erd\H{o}s distance problem in vector spaces over finite fields},  Transactions of the American Mathematical Society, \textbf{359}(12) (2007), 6127--6142.
\bibitem{Kra}
B. Kra, J. Moreira, F. K. Richter, and D. Robertson, \textit{The density finite sums theorem}, Inventiones Mathematicae, (2025), 1--31.

\bibitem{book3}
L. Kuipers and H. Niederreiter, \textit{Uniform distribution of sequences}, Courier Corporation, 2012.

\bibitem{JM}
J. Maynard, \textit{Simultaneous small fractional parts of polynomials}, Geometric and Functional Analysis, \textbf{31}(1) (2021), 150--179.

\bibitem{munkres2000}
J. R. Munkres, \textit{Topology} (2nd ed.), Prentice Hall, 2000.

\bibitem{PintzSteigerSzemeredi1988}
J.~Pintz, W.~L.~Steiger, and E.~Szemer\'edi,
\emph{On sets of natural numbers whose difference set contains no squares},
Journal of the London Mathematical Society, \textbf{s2-37}(2) (1988), 219--231.
\bibitem{Raimi}  R. Raimi, \textit{Translation properties of finite partitions of the positive integers}, Fundamenta Mathematicae, \textbf{61}(3) (1967), 253--256.



\bibitem{Sarkozy1978}
A.~S\'ark\"ozy,
\emph{On difference sets of sequences of integers. I},
Acta Mathematica Academiae Scientiarum Hungaricae, \textbf{31}(1--2) (1978), 125--149.

\bibitem{schmidt}
W. M. Schmidt, \textit{Small fractional parts of polynomials}, Volume 32, American Mathematical Society, 1977.

\bibitem{Sze} E. Szemer\'{e}di, \textit{On sets of integers containing no $k$ elements in arithmetic progression}, Acta Arithmetica, \textbf{27} (1975), 199--245.


\bibitem{vanderWaerden1927}
B.~L.~Van der Waerden,
\emph{Beweis einer Baudetschen Vermutung},
Nieuw Archief voor Wiskunde, \textbf{15} (1927), 212--216.
\bibitem{Hweyl1} H. Weyl, \textit{Über ein Problems aus dem Gebiete der diophantischen Approximationen}, Göttingen, Mathematisch-Physikalische Klasse, \textbf{1914} (1914), 234--244.

\bibitem{Hweyl2}
H. Weyl, \textit{Über die Gleichverteilung von Zahlen mod. Eins}, Mathematische Annalen, \textbf{77}(3) (1916), 313--352.



\end{thebibliography}
\end{document}